\documentclass[12pt]{amsart}
\usepackage{latexsym,amssymb,amscd,eqname}
\def\B'c{{\mathcal{B'}}}
\def\U'c{{\mathcal{U'}}}

\def\opn#1#2{\def#1{\operatorname{#2}}} % to make operators
\opn\chara{char}
\opn\length{\ell}
\opn\projdim{proj\,dim}
\opn\injdim{inj\,dim}
\opn\ini{in}
\opn\rank{rank}
\opn\depth{depth}
\opn\height{ht}
\opn\embdim{emb\,dim}
\opn\codim{codim}

\opn\Tr{Tr}
\opn\bigrank{big\,rank}
\opn\superheight{superheight}\opn\lcm{lcm}
\opn\trdeg{tr\,deg}%
\opn\reg{reg}
\opn\lreg{lreg}
\opn\set{set}
\opn\supp{Supp}
\opn\shad{Shad}
\opn\del{del}
\opn\succ{succ}
\opn\pred{pred}
\opn\div{div}
\opn\Div{Div}
\opn\cl{cl}
\opn\Cl{Cl}
\opn\Spec{Spec}
\opn\Supp{Supp}
\opn\supp{supp}
\opn\Sing{Sing}
\opn\Ass{Ass}
\opn\Min{Min}
\opn\Mon{Mon}
\opn\Ann{Ann}
\opn\Rad{Rad}
\opn\Soc{Soc}
\opn\Ker{Ker}
\opn\Coker{Coker}
\opn\Im{Im}
\opn\Hom{Hom}
\opn\Tor{Tor}
\opn\Ext{Ext}
\opn\End{End}
\opn\Aut{Aut}
\opn\id{id}

\opn\nat{nat}
\opn\GL{GL}
\opn\SL{SL}
\opn\mod{mod}
\opn\ord{ord}
\opn\aff{aff}
\opn\con{conv}
\opn\relint{relint}
\opn\st{st}
\opn\lk{lk}
\opn\cn{cn}
\opn\core{core}
\opn\vol{vol}
\opn\gr{gr}

\def\pot#1#2{#1[\kern-0.28ex[#2]\kern-0.28ex]}

\opn\dirlim{\underrightarrow{\lim}}
\opn\invlim{\underleftarrow{\lim}}
\def\pnt{{\raise0.5mm\hbox{\large\bf.}}}

\def\twoline#1#2{\aoverb{\scriptstyle {#1}}{\scriptstyle {#2}}}
\newcommand{\aoverb}[2]{{\genfrac{}{}{0pt}{1}{#1}{#2}}}
\def\Implies{\ifmmode\Longrightarrow \else
     \unskip${}\Longrightarrow{}$\ignorespaces\fi}
\def\implies{\ifmmode\Rightarrow \else
     \unskip${}\Rightarrow{}$\ignorespaces\fi}
\def\iff{\ifmmode\Longleftrightarrow \else
     \unskip${}\Longleftrightarrow{}$\ignorespaces\fi}

\let\:=\colon
\newtheorem{Theorem}{Theorem}[section]
\newtheorem{Lemma}[Theorem]{Lemma}
\newtheorem{Corollary}[Theorem]{Corollary}
\newtheorem{Proposition}[Theorem]{Proposition}

\newtheorem{Definition}[Theorem]{Definition}

\let\epsilon=\varepsilon
\let\phi=\varphi
\let\kappa=\varkappa
\title{Classes of sequentially Cohen--Macaulay squarefree monomial ideals}
\author{Oana Olteanu}
	
\address{Faculty of Mathematics and Computer Science, Ovidius University, Bd.\ Mamaia 124,
 900527 Constanta, Romania,} \email{olteanuoanastefania@gmail.com}

\begin{document}

\maketitle

\begin{abstract}
We compute the minimal primary decomposition for completely squarefree lexsegment ideals. We show that critical squarefree monomial ideals are sequentially Cohen--Macaulay. As an application, we give a complete characterization of the completely squarefree lexsegment ideals which are sequentially Cohen--Macaulay and we also derive formulas for some homological invariants of this class of ideals. 

Keywords: squarefree lexsegment ideals, primary decomposition, sequentially Cohen--Macaulay ideals.\\ 

MSC: 13H10, 13A15.

\end{abstract}

\section*{Introduction}
In analogy with the notion of nonpure shellable simplicial complex introduced by Bj\"orner and Wachs, Stanley \cite[Section III.2]{St} defined the concept of sequentially Cohen--Macaulay module, a nonpure generalization of Cohen--Macaulayness. A simplicial complex $\Delta$ is \textit{sequentially Cohen--Macaulay} if all its pure skeletons are Cohen--Macaulay. It is known (\cite{HH}) that the associated Stanley--Reisner ideal $I_\Delta$ is sequentially Cohen--Macaulay, that is, $S/I_{\Delta}$ is a sequentially Cohen--Macaulay module, if and only if $I^\vee=I_{\Delta^\vee}$ is componentwise linear, which means that for all $d\geq 0$, the ideal $I^\vee_{\langle d\rangle}$ generated by all degree $d$ elements in $I^\vee$ has a linear resolution. Here we denoted as usual by $\Delta^\vee$ the Alexander dual of $\Delta$.

The critical ideals have been studied in \cite{MH}. We consider in this paper the critical squarefree monomial ideals. It turns out that they are componentwise linear (Corollary \ref{critical comp lin}). This property will be useful in studying Alexander duals of squarefree lexsegment ideals.  
 
Let $S=k[x_1,\ldots,x_n]$ be the polynomial ring in $n$ variables over a field $k$. We order the monomials in $S$ lexicographically with $x_1>x_2>\ldots >x_n$. For an arbitrary integer $q\geq 2$, we denote by $\Mon^s_q(S)$ the set of all squarefree monomials of degree $q$ in $S$. A \textit{squarefree lexsegment set of degree $q$} determined by the monomials $u,\ v\in \Mon_q^s(S)$ is a subset of $\Mon^s_q(S)$ of the form $L(u,v)=\{w\in\Mon_q^s(S)\ :\ u\geq_{lex}w\geq_{lex} v\}$. An ideal generated by a squarefree lexsegment set is called \textit{squarefree lexsegment ideal}. In particular, one can define \textit{initial} and \textit{final squarefree lexsegment sets} to be sets of the form $L^i(v)=\{w\in\Mon_q^s(S)\ :\ w\geq_{lex} v\}$, respectively $L^f(u)=\{w\in\Mon_q^s(S)\ :\ u\geq_{lex}w\}$ and \textit{initial} and \textit{final squarefree lexsegment ideals}, accordingly. 

The concept of squarefree lexsegment ideal was introduced in \cite{AHH} by Aramova, Herzog and Hibi, where the notion was associated with the nowadays concept of initial squarefree lexsegment ideals, but it was also studied in \cite{ADH}, \cite{AHH}, \cite{B} and \cite{BoS}.  

A lexsegment $L$ is called \textit{completely lexsegment} if all the iterated shadows of $L$ are again lexsegments. By the \textit{shadow} of a set $T$ of monomials in $S$ we mean the set $\shad(T)=\{wx_i: w\in T, x_i\nmid w,1\leq i\leq n\}$. The \textit{$i$-th shadow} is defined recursively by $\shad^i(T)=\shad(\shad^{i-1}(T))$. Initial squarefree lexsegment ideals are examples of completely squarefree lexsegment ideals. In \cite{B}, Bonanzinga proved a persistence theorem for squarefree lexsegment ideals. Moreover, the squarefree lexsegment ideals with a linear resolution were characterized in \cite{B}.

In this paper we are interested in studying the completely squarefree lexsegment ideals. In Section 2, we explicitly compute the minimal primary decomposition for initial and final squarefree lexsegment ideals. Using the fact that any completely squarefree lexsegment ideal $(L(u,v))$ can be written as the intersection of $(L^i(v))$ with $(L^f(u))$, we are able to compute the standard primary decomposition for this class of squarefree lexsegment ideals. As first consequences, we obtain formulas for the Krull dimension of $S/I$, in the case that $I$ is a completely squarefree lexsegment ideal and its multiplicity.

As an application of the minimal primary decomposition, in Section 3 we characterize all the completely squarefree lexsegment ideals which are sequentially Cohen--Macaulay. 

In the last section we give bounds for $\depth(S/I)$, where $I$ is an arbitrary squarefree lexsegment ideal. 
	\[
\]
\textbf{Acknowledgment.} The author would like to thank Professor J\"urgen Herzog and Professor Viviana Ene for valuable discussions and comments during the preparation of this paper, and suggestions for improvement.  
	\[
\]
\section{Preliminaries}

Let $S=k[x_1,\ldots,x_n]$ be the polynomial ring in $n$ variables over a field $k$.

For an integer $q\geq 2$, let $\Mon_q^s(S)$ be the set of all squarefree monomials of degree $q$ in $S$. We consider the lexicographical order on the monomials in $S$ with $x_1>\ldots>x_n$. 

Given two monomials $u,\ v\in\Mon_q^s(S)$, the set $L(u,v)=\{w\in\Mon_q^s(S)\ :\ u\geq_{lex}w\geq_{lex} v\}$ is called the \textit{squarefree lexsegment set} determined by $u$ and $v$. In particular, the set $L^i(v)=\{w\in\Mon_q^s(S)\ :\ w\geq_{lex} v\}$ is the \textit{initial squarefree lexsegment set} of $v$ and $L^f(u)=\{w\in\Mon_q^s(S)\ :\ u\geq_{lex}w\}$ is called the \textit{final squarefree lexsegment set} of $u$. An \textit{(initial, final) squarefree lexsegment ideal} is the squarefree monomial ideal generated by an (initial, final) squarefree lexsegment set.

For a squarefree monomial ideal $I\subset S$, we may consider the simplicial complex $\Delta$ on the vertex set $[n]$ such that the Stanley--Reisner ideal of $\Delta$ is $I$. The standard primary decomposition of a Stanley--Reisner ideal can be written by looking at the facets of the simplicial complex.

\begin{Proposition}\cite{St}\label{St Reisner ideal decomposition}
Let $\Delta$ be the simplicial complex on the vertex set $[n]$ and $k$ a field. Then
$$I_{\Delta}=\bigcap\limits_{F\in\mathcal F(\Delta)}P_{F^c},$$
where $P_{F^c}$ are the prime ideals generated by all the variables $x_i$ such that $i\notin F$. 
\end{Proposition}

For every integer $q\in[n]$, we denote by $I_{n,q}\subset S$ the squarefree monomial ideal generated by all the squarefree monomials of degree $q$. Then $I_{n,q}$ is the Stanley--Reisner ideal of $\Delta$, where $\Delta$ is generated by all the subsets of $[n]$ of cardinality $q-1$. Applying Proposition \ref{St Reisner ideal decomposition}, we obtain 
$$I_{n,q}=\bigcap\limits_{F\in\mathcal F(\Delta)}P_{F^c}=\bigcap\limits_{|F|=q-1}P_{F^c}.$$

It is known that the multiplicity of the Stanley--Reisner ring of a simplicial complex $\Delta$ can be expressed in terms of the $f-$vector of $\Delta$.

\begin{Lemma}\cite{BH}\label{multiplicity}
Let $\Delta$ be a simplicial complex of dimension $d-1$ and $k$ be a field. Then the multiplicity of the Stanley--Reisner ring of $\Delta$ is $e(k[\Delta])=f_{d-1}$. 
\end{Lemma}

%\begin{Definition}\rm
%Let $\Delta$ be a simplicial complex on the vertex set $[n]$ of dimension $\dim\Delta=d-1$. For all $0\leq i\leq d-1$ we define the \textit{$i$-th skeleton} of $\Delta$ to be the simplicial complex
%$$\Delta^{(i)}=\{F\in\Delta\ :\ |F|\leq i+1\}.$$
%\end{Definition}

For a squarefree monomial ideal $I\subset S$, one may use the following criterion due to Hibi (\cite{H}) to compute $\depth(S/I)$. We recall that the $i-$th skeleton of a simplicial complex $\Delta$ is $\Delta^{(i)}=\{F\in \Delta: \dim(F)\leq i\}$.

\begin{Lemma}\label{CM skeleton}\cite{H}
Let $\Delta$ be a simplicial complex of dimension $d-1$ and $\Delta^{(i)}$ the $i-$th skeleton, for all $0\leq i\leq d-1$. Then:
\begin{enumerate}
	\item [(a)] If $\Delta$ is Cohen--Macaulay, then $\Delta^{(i)}$ is Cohen--Macaulay, for all $0\leq i\leq d-1$;
	\item [(b)] $\depth(k[\Delta])=\max\{i+1:k[\Delta^{(i)}]\mbox{ is Cohen--Macaulay }\}$.
\end{enumerate}
\end{Lemma}

An important class of squarefree monomial ideals consists of edge ideals. 

\begin{Definition}\rm
The \textit{edge ideal} $I(G)$ associated with the graph $G=(V(G),E(G))$ is the squarefree monomial ideal in $S$ generated by all the monomials $x_ix_j$, with $\{i,j\}\in E(G)$.
\end{Definition}

Let $G=(V(G),E(G))$ be a graph. A subset $A\subset V(G)$ is called a \textit{minimal vertex cover} of $G$ if every edge of $G$ is incident to one vertex in $A$ and there is no proper subset of $A$ with this property.   

There is a strong relation between the minimal vertex covers of a graph and the minimal prime ideals of the edge ideal.

\begin{Proposition}\cite{Vi}\label{min vertex cover}
Let $G$ be a graph on the vertex set $[n]$. If $P\subset S$ is the ideal generated by $A=\{x_{i_1},\ldots,x_{i_r}\}$, then $P$ is a minimal prime ideal of $I(G)$ if and only if $A$ is a minimal vertex cover of $G$.
\end{Proposition}

\section{Primary decomposition for completely squarefree lexsegment ideals}

The goal of this section is to determine the minimal primary decomposition of a completely squarefree lexsegment ideal. By completely squarefree lexsegment ideal we mean a  squarefree lexsegment ideal whose squarefree shadow remains a squarefree lexsegment ideal. The squarefree shadow of a set T of monomials in $S=k[x_1,\ldots,x_n]$ is the set $\shad(T)=\{x_im:m\in T, x_i\nmid m,1\leq i\leq n\}$. Initial squarefree lexsegment ideals are an example of completely squarefree lexsegment ideals. Moreover, it is known that any completely squarefree lexsegment ideal $I=(L(u,v))$ can be written as the intersection of the initial squarefree lexsegment ideal $(L^i(v))$ with the final squarefree lexsegment ideal $(L^f(u))$, by \cite{B}. Thus we will firstly determine the minimal primary decomposition for initial and final squarefree lexsegment ideals. 

Let $I=(L^i(v))\subset k[x_1,\ldots,x_n]$ be an initial squarefree lexsegment ideal, where $v=x_{j_1}\cdots x_{j_q}$. We may assume that $j_1\geq 2$. Otherwise, if $j_1=1$, then $I=(x_1)\cap (L(x_2\cdots x_q,v/x_1))$ and the problem reduces to compute the minimal primary decomposition of an initial squarefree lexsegment ideal in a fewer number of variables. 

\begin{Theorem}\label{prim dec initial}
Let $I\subset k[x_1,\ldots,x_n]$ be the initial squarefree lexsegment ideal generated in degree $q$, determined by the monomial $v=x_{j_1}\cdots x_{j_q}$, with $2\leq j_1<\ldots<j_q\leq n$. Consider the sets $A_t=[j_t]\setminus\{j_1,\ldots,j_{t-1}\}$, for $1\leq t\leq q$. Then $I$ has the minimal primary decomposition of the form:
$$I=\left(\bigcap\limits_{t=1}^{q}(x_i\ :\ i\in A_t)\right)\cap\left(\bigcap\limits_\twoline{F\subset[n],\ |F|=q-1}{F\cap A_t\neq\emptyset,\ \forall t}P_{F^c}\right).$$  
\end{Theorem}
  
\begin{proof}
Let $\Delta$ be the simplicial complex on the vertex set $[n]$ such that $I=I_{\Delta}$. We show that the facets of $\Delta$ are exactly the sets $[n]\setminus A_t$, $1\leq t\leq q$, together with the sets $F\subset [n]$ with $|F|=q-1$ and such that $F\cap A_t\neq\emptyset$ for all $t$. By applying Proposition \ref{St Reisner ideal decomposition} we next get the desired formula.

In the first place, we observe that all the sets $G\subset [n]$ with $|G|=q-1$ are faces of $\Delta$ since $x_G\notin I$. Therefore, the facets of $\Delta$ have the cardinality at least $q-1$. Secondly, we note that $G$ is a facet of $\Delta$ if and only if $x_G\notin I$ and $x_{G\cup\{i\}}\in I$, for all $i\in [n]\setminus G$.

Let $G$ be a subset of the set $[n]$ with $|G|=q-1$. Then $G$ is a facet of $\Delta$ if and only if $x_{G\cup\{i\}}\in I$, for all $i\in[n]\setminus G$, which is equivalent to $x_Gx_{\max([n]\setminus G)}\geq_{lex} v$. We show that this last condition is equivalent to $G\cap A_t\neq\emptyset$, for all $t$. It is obvious that if $G\cap A_t=\emptyset$, for some $1\leq t\leq q$, then $G\subset \{j_1,\ldots ,j_{t-1}\}\cup\{j_t+1,\ldots,n\}$, hence $x_Gx_{\max([n]\setminus G)}<_{lex} v$, a contradiction. Conversely, let $G\cap A_t\neq\emptyset$, for all $1\leq t\leq q$ and assume that $x_Gx_{\max([n]\setminus G)}<_{lex} v$. Then $x_G\leq_{lex}v/x_{\max(v)}=x_{j_1}\cdots x_{j_{q-1}}$, that is, either $x_G=x_{j_1}\cdots x_{j_{q-1}}$, which is imposible since $G\cap A_q\neq\emptyset$, or $x_G<_{lex}x_{j_1}\cdots x_{j_{q-1}}$, which implies that there exists $1\leq t\leq q-1$ such that $G\cap A_t=\emptyset$, again a contradiction.  

Next, we look at the facets of cardinality greater than or equal to $q$. Let $G\subset[n]$ be such a facet and $x_G=x_{l_1}\cdots x_{l_r}$, with $r\geq q$. It is clear that if $l_1<j_1$, then $x_G\in I$, thus $l_1\geq j_1$. The only facet with $l_1>j_1$ is $G=\{j_1+1,\ldots, n\}=[n]\setminus A_1$. Assume now that $l_1=j_1$. Then we obtain $l_2\geq j_2$. The only facet with $l_2>j_2$ is $G=\{j_1,j_2+1,\ldots, n\}=[n]\setminus A_2$. By applying this argument step by step, it follows that the facets of cardinality greater than or equal to $q$ are exactly $[n]\setminus A_t$, with $t=1,\ldots, q$.
\end{proof}

%It is known that any Cohen--Macaulay simplicial complex is pure. 

\begin{Corollary}
Let $I=(L^i(v))\subset S$ be an initial squarefree lexsegment ideal, where $v=x_{j_1}\cdots x_{j_q}$, and $\Delta$ be the simplicial complex on the vertex set $[n]$ such that $I=I_{\Delta}$. The simplicial complex $\Delta$ is pure if and only if $v=x_{n-q+1}\cdots x_n$. Moreover, $\Delta$ is a pure simplicial complex if and only if $\Delta$ is Cohen--Macaulay.
\end{Corollary}

\begin{proof}
It is known that any Cohen--Macaulay simplicial complex is pure. 

The simplicial complex $\Delta$ is pure if all its facets have the same dimension. Using Proposition \ref{St Reisner ideal decomposition}, we obtain that $n-j_1=n-j_2+1=\ldots=n-j_q+q-1=q-1$. Therefore the monomial $v$ must have the form $v=x_{n-q+1}\cdots x_n$. 

Since the ideal generated by all the squarefree monomials of degree $q$ is Cohen--Macaulay, it results that $\Delta$ is a pure simplicial complex if and only if $\Delta$ is Cohen--Macaulay.
\end{proof}

The following result appears in \cite{AHH}, but we present it as a consequence of the minimal primary decomposition.

\begin{Corollary}\label{dim}
In the hypothesis of Theorem \ref{prim dec initial}, the dimension of the Stanley--Reisner ring of $\Delta$ is $n-j_1$.
\end{Corollary}

\begin{proof}
The height of the Stanley--Reisner ideal associated to $\Delta$ is $\height I=\min\{j_1,j_2-1,\ldots,j_q-(q-1),n-q+1\}=j_1$. Hence $\dim k[\Delta]=\dim(S/I)=n-j_1$.  
\end{proof}

\begin{Corollary}\label{depth}
If $I\subset S$ is an initial squarefree lexsegment ideal generated in degree $q$, then $\depth(S/I)=q-1$.
\end{Corollary}

\begin{proof}
The claim is obvious if $I= I_{n,q}$. We assume that $I\neq I_{n,q}$. Since $I$ is generated by squarefree monomials of degree $q$, we have that any subset $F$ of $[n]$, with $|F|\leq q-1$, is a face of $\Delta$. This implies that $I_{\Delta^{(q-2)}}$ is Cohen--Macaulay, because it is generated by all the monomials of degree $q$. By Lemma \ref{CM skeleton} we have that $\depth (S/I)\geq q-1$.  

The facets of $\Delta$ are of cardinality $n-j_1,n-j_2+1,\ldots,n-j_q+q-1,q-1$, with $n-j_1\geq n-j_2+1\geq\ldots\geq n-j_q+q-1>q-1$. It results that $\Delta^{(q-1)}$ is not a pure simplicial complex, therefore it is not Cohen--Macaulay. By Lemma \ref{CM skeleton}, we obtain $\depth (S/I)\leq q-1$. 
\end{proof}

%Let $G$ be the graph with the edge ideal $I(G)$, where $I(G)$ is the initial squarefree lexsegment ideal generated in degree $2$, determined by the monomial $v=x_{j_1}x_{j_2}$. We have    
%$$I(G)=(x_i\ :\ i\in A_1)\cap(x_i\ :\ i\in A_2)\cap\left(\bigcap\limits_\twoline{|F|=1,\ F\subset[n]}{F\cap A_t\neq\emptyset,\ \forall t}P_{F^c}\right),$$  
%where $$A_1=[j_1],\ A_2=[j_2]\setminus\{j_1\}.$$
As a consequence of the Proposition \ref{min vertex cover}, we obtain: 

\begin{Corollary}\label{min vertex cover initial}
Let $G$ be the graph with the edge ideal $I(G)$, where $I(G)$ is the initial squarefree lexsegment ideal generated in degree $2$, determined by the monomial $v=x_{j_1}x_{j_2}$. Then 
$$I(G)=(x_i\ :\ i\in A_1)\cap(x_i\ :\ i\in A_2)\cap\left(\bigcap\limits_{i=1}^{j_1-1}P_{[n]\setminus\{i\}}\right),$$
where $A_1=[j_1],\ A_2=[j_2]\setminus\{j_1\}$. The sets $A_1,\ A_2$ and $[n]\setminus\{i\}$, for $i=1,\ldots ,j_1-1$, are the minimal vertex covers of $G$.
\end{Corollary}

\begin{Corollary}
Let $I\subset S$ be the initial squarefree lexsegment ideal generated in degree $q$, determined by the monomial $v=x_{j_1}\cdots x_{j_q}$, with $2\leq j_1<\ldots<j_q\leq n$, $v\neq x_{n-q+1}\cdots x_n$, and $\Delta$ the simplicial complex with the Stanley--Reisner ideal $I$. If $s$ is the unique integer such that $j_i=j_1+(i-1)$, for all $1\leq i< s$ and $j_s\geq j_1+s$, then the multiplicity of $k[\Delta]$ is $e(k[\Delta])=s-1$.
\end{Corollary}

\begin{proof}
By Lemma \ref{multiplicity} and  Corollary \ref{dim}, the multiplicity of $k[\Delta]$ is $e(k[\Delta])=f_{n-j_1-1}$. Therefore, we need to determine the number of facets of $\Delta$ of cardinality $n-j_1$. This is equivalent, by Theorem \ref{prim dec initial}, to determine the number of minimal prime ideals which contains $I$, with $\height(p)=j_1$ and $j_1\neq n-q+1$. In the notations of Theorem \ref{prim dec initial}, if $s$ is the unique integer such that $j_i=j_1+(i-1)$, for all $1\leq i< s$ and $j_s\geq j_1+s$, then $|A_i|=j_i-(i-1)=j_1$, for all $1\leq i\leq s-1$ and $|A_i|=j_i-(i-1)>j_1$, for $i\geq s$. Therefore, $e(k[\Delta])=s-1$.   
\end{proof}

%\begin{Proposition}\label{shad-final}\cite{B}
%Let $u$ be a monomial of degree $q$, with $q<n-2$ and $I=(L^f(u))$ the final squarefree lexsegment ideal. Then the following conditions are equivalent:
%\begin{itemize}
%	\item [(a)] $I$ is completely lexsegment;
%	\item [(b)] $\shad(I)$ is a lexsegment;
%	\item [(c)] $u\geq_{lex}x_3\cdots x_{q+2}.$
%\end{itemize}
%\end{Proposition}
%In \cite{B} is given the characterization of the final lexsegment ideals which are completely.

Next, we discuss the case of a final squarefree lexsegment ideals $(L^f(u))$, where $u$ is a monomial in $S$. Note that we may reduce to the hypothesis $x_1\mid u$. Indeed, otherwise, $x_1,\ldots,x_{\min(u)-1}$ are regular on $S/I$, hence they do not belong to any associated prime of $I$. Therefore, computing the primary decomposition of $I$ in $S$ is equivalent to compute the primary decomposition of $I\cap k[x_{\min(u)},\ldots,x_n]$.  

\begin{Proposition}\label{deg-dual}
Let $I=(L^f(u))\subset S$ be the final squarefree lexsegment ideal determined by the monomial $u=x_1x_{i_2}\cdots x_{i_q}$, $2\leq i_2<\ldots<i_q\leq n$. We denote by $\Delta$ the simplicial complex with $I=I_{\Delta}$. Then the Stanley--Reisner ideal of the Alexander dual of $\Delta$, $I^\vee$, is generated in degree $n-q$ and $n-q+1$.
\end{Proposition}

\begin{proof}
Since $u=x_1x_{i_2}\cdots x_{i_q}$, one may easily check that $\shad(I)=\Mon^s_{q+1}(S)$. Moreover, $\shad^s(I)=\Mon^s_{q+s}(S)$, for all $s\geq 1$. This implies that all the squarefree monomials of degree greater than or equal to $q+1$ belong to $I$. Hence, $I^{\vee}$ is generated in degree greater than or equal to $n-q$. 

On the other hand, all the monomials $x_G$, with $|G|\leq q-1$ do not belong to $I$, thus all the monomials $x_{[n]\setminus G}$ of degree greater than or equal to $n-q+1$ belong to $I^\vee$.

Therefore, $I^\vee$ is generated in degree $n-q$ and $n-q+1$.     
\end{proof}

The following result gives us the minimal primary decomposition of a final squarefree lexsegment ideal. %For a monomial $w\in \Mon(S)$, we denote by $\supp(w)=\{x_i\ :\ x_i\mid w\}$.

\begin{Theorem}\label{prim dec final}
Let $I\subset S$, $I\neq I_{n,q}$, be the final squarefree lexsegment ideal generated in degree $q$, determined by the monomial $u=x_1x_{i_2}\cdots x_{i_q}$, $2\leq i_2<\ldots<i_q\leq n$. Denote by $F=\{i_2,\ldots,i_q\}$ and by $x_F=\prod\limits_{i\in F}x_i$. Then $I$ has the minimal primary decomposition of the form:
$$I=\left(\bigcap\limits_\twoline{G\subset[n],\ |G|=n-q+1}{x_G\geq_{lex}x_{F^c}}P_G\right)\cap\left(\bigcap\limits_\twoline{G\subset[n]\setminus\{1\},\ |G|=n-q+1}{x_{G\setminus\min(G)}\geq_{lex}x_{F^c\setminus\{1\}}}P_G\right)\cap\left(\bigcap\limits_\twoline{G\subset[n],\ |G|=n-q}{x_{F^c\setminus\{1\}}>_{lex}x_G}P_G\right).$$
\end{Theorem}

\begin{proof}
Let $\Delta$ be the simplicial complex on the vertex set $[n]$ such that $I=I_{\Delta}$. By Proposition \ref{deg-dual}, we know that all the facets of $\Delta$ have cardinality $q-1$ or $q$. A facet $H$ of $\Delta$ is characterized by the condition $x_H\notin I$ and $x_{H\cup\{i\}}\in I$, for all $i\in [n]\setminus H$. 

Let $H$ be a subset of $[n]$ of cardinality $q$. Note that, if $|H|=q$, then $x_Hx_i\in \Mon_{q+1}^s(S)=\shad(I)\subset I$, for all $i\in [n]\setminus H$, hence $H$ is a facet of $\Delta$ if and only if $x_H\notin I$. We have $x_H\notin I$ if and only if $x_H>_{lex} u=x_1x_F$, that is $x_{H^c}<_{lex}x_{F^c\setminus\{1\}}$. Therefore, by denoting $G=H^c$, we get the last family of minimal prime ideals in the formula of the theorem.

Now, we look at the facets of cardinality $q-1$. Let $H$ be a subset of $[n]$, with $|H|=q-1$. Then $x_H\notin I$, thus it remains to characterize the sets $H$ with $x_{H\cup\{i\}}\in I$, for all $i\in [n]\setminus H$. This is equivalent to $x_Hx_{\min([n]\setminus H)}\in I$, that is $u=x_1x_F\geq_{lex}x_Hx_{\min([n]\setminus H)}$. We distinguish two cases:

\textit{Case 1:} If $\min([n]\setminus H)=1$, then $x_1x_F\geq_{lex}x_Hx_{\min([n]\setminus H)}$ is equivalent to $x_F\geq _{lex} x_H$, that is $x_{H^c}\geq_{lex} x_{F^c}$. Therefore, in this case, we get the first family of minimal prime ideals of $I$.

\textit{Case 2:} If $\min([n]\setminus H)>1$, then $1\in H$. By taking the complements in $x_1x_F\geq_{lex}x_Hx_{\min([n]\setminus H)}$, we obtain $x_{([n]\setminus H)\setminus\{\min([n]\setminus H)\}}\geq _{lex} x_{F^c\setminus\{1\}}$, thus, by setting $G=[n]\setminus H$, we get the second family of minimal prime ideals.
\end{proof}

The minimal primary decomposition allows us to compute the Krull dimension of the quotient ring $S/I$.
 
\begin{Corollary}\label{dim final}
If $I=(L^f(u))\subset S$ is the final squarefree lexsegment ideal determined by the monomial $u$, where $u=x_1x_{i_2}\cdots x_{i_q}$, $2\leq i_2<\ldots<i_q\leq n$, then $\dim(S/I)=q$.
\end{Corollary}

\begin{Corollary}
Let $G$ be the graph with the edge ideal $I(G)$, where $I(G)$ is the final squarefree lexsegment ideal determined by the monomial $u=x_1x_{i_2}$, with $i_2>2$. Then 
$$I(G)=\left(\bigcap\limits_{s\geq i_2}P_{[n]\setminus\{s\}}\right)\cap\left(\bigcap\limits_{s<i_2}P_{[n]\setminus\{1,s\}}\right).$$
The minimal vertex covers of $G$ are the sets $[n]\setminus\{s\}$, with $s\geq i_2$ together with the sets $[n]\setminus\{1,s\}$, with $s<i_2$. 
\end{Corollary}

%Here we assumed that $i_2>2$. For the case $i_2=2$, the minimal vertex covers are given in Corollary \ref{min vertex cover}, beeing an initial squarefree lexsegment ideal. 

\begin{Corollary}
Let $I\subset S$ be the final squarefree lexsegment ideal determined by the monomial $u=x_1x_{i_2}\cdots x_{i_q}=x_1x_F$, $2\leq i_2<\ldots<i_q\leq n$ and $\Delta$ the simplicial complex with the Stanley-Reisner ideal $I$. Then the multiplicity of $k[\Delta]$ is $e(k[\Delta])=|\{x_G: x_{F^c\setminus\{1\}}>_{lex}x_G, |G|=n-q\}|$.
\end{Corollary}

\begin{proof}
The multiplicity of $k[\Delta]$ is $e(k[\Delta])=f_{q-1}$, by Lemma \ref{multiplicity} and Corollary \ref{dim final}. The number of facets of $\Delta$ of cardinality $q$, that is $f_{q-1}$, is the same with the number of minimal prime ideals of $I$ of $\height(p)=n-q$. By Theorem \ref{prim dec final}, if we denote by $t$ the cardinality of the set $\{x_G: x_{F^c\setminus\{1\}}>_{lex}x_G, |G|=n-q\}$, then the multiplicity of $k[\Delta]$ is $e(k[\Delta])=t$.
\end{proof}

We will use Lemma \ref{CM skeleton} in order to compute the depth of a final squarefree lexsegment ideal.

%\begin{Lemma}\label{CM skeleton}
%Let $\Delta$ be a simplicial complex of dimension $d-1$ and $\Delta^{(i)}$ the $i-$th skeleton, for all $0\leq i\leq d-1$. Then:
%\begin{enumerate}
%	\item [(a)] If $\Delta$ is Cohen-Macaulay, then $\Delta^{(i)}$ is Cohen-Macaulay, for all $0\leq i\leq d-1$
%	\item [(b)] $\depth(k[\Delta])=\max\{i+1\mbox{ : such that }k[\Delta^{(i)}]\mbox{ is Cohen-Macaulay }\}$.
%\end{enumerate}
%\end{Lemma}

\begin{Corollary}\label{depth finale}
If $I=(L^f(u))\subset S$ is the final squarefree lexsegment ideal determined by the monomial $u=x_1x_{i_2}\cdots x_{i_q}$, $2\leq i_2<\ldots<i_q\leq n$, then $\depth(S/I)=q-1$.
\end{Corollary}

\begin{proof}
We may consider $u\neq x_1\cdots x_q$. Let $\Delta$ be the simplicial complex on the vertex set $\{x_1,\ldots,x_n\}$ with the Stanley-Reisner ideal $I_{\Delta}=I$. %Since $I$ is generated by squarefree monomials of degree $q$, we have that any subset $\sigma$ of $\{x_1,\ldots,x_n\}$ of cardinality $|\sigma|\leq q-1$ is a face of $\Delta$. This implies that the Stanley-Reisner ideal associated to the $(q-2)-$skeleton of $\Delta$, $I_{\Delta^{(q-2)}}$, is generated by all the monomials of degree $q$, therefore $I_{\Delta^{(q-2)}}=I_{n,q}$. Since $I_{n,q}$ is Cohen-Macaulay, we have that the Stanley-Reisner ring $k[\Delta^{(q-2)}]$ is Cohen-Macaulay. By Lemma \ref{CM skeleton}, we have that $\depth (S/I)\geq q-2+1=q-1$.
As in the first part of the proof of Corollary \ref{depth}, we have that $\depth (S/I)\geq q-1$.

We prove that $\Delta ^{(q-1)}$ is not pure. Indeed, because $u\neq x_1\cdots x_q$, we have that $\{x_1,\ldots,x_q\}$ is a facet of $\Delta^{(q-1)}$. We consider $\tau=\{x_{n-q+2},\ldots,x_n\}$, $|\tau|=q-1$. Thus $\tau$ is a maximal face of $\Delta^{(q-1)}$ because all the monomials $x_ix_{n-q+2}\cdots x_n\in I$, for all $1\leq i\leq n-q+1$. Therefore $\Delta ^{(q-1)}$ is not pure. This implies that $k[\Delta ^{(q-1)}]$ is not Cohen--Macaulay and we get $\depth(S/I)\leq q-1$.
\end{proof}

Let $I=(L(u,v))$ be a completely squarefree lexsegment ideal. By \cite{B}, one may write $I=(L^i(v))\cap(L^f(u))$. This allows us to determine the standard primary decomposition of an arbitrary completely squarefree lexsegment ideal.

\begin{Theorem}\label{prim dec completely}
Let $I\subset S$ be a completely squarefree lexsegment ideal generated in degree $q$, determined by the monomials $u=x_1x_F$ and $v=x_{j_1}\cdots x_{j_q}$, $2\leq j_1<\ldots<j_q\leq n$. Then the minimal primary decomposition of $I$ is the following:
$$I=\bigcap\limits_{|A_t|\leq n-q}P_{A_t}\cap\bigcap\limits_\twoline{|A_t|=n-q+1}{u/x_1\geq_{lex} v/x_{j_t}}P_{A_t}\cap\bigcap\limits_\twoline{G\subset[n],\ |G|=q-1,\ G\cap A_t\neq\emptyset,\ \forall t}{u/x_1\geq_{lex} x_G}P_{G^c}\cap\bigcap\limits_\twoline{P\in \Min(L^f_{S}(u))}{\height P=n-q} P,$$
if $x_2\nmid u$, and
$$I=\bigcap\limits_{|A_t|\leq n-q}P_{A_t}\cap\bigcap\limits_\twoline{|A_t|=n-q+1}{u/x_1\geq_{lex} v/x_{j_t}}P_{A_t}\cap\bigcap\limits_\twoline{G\subset[n]\setminus\{1\},\ |G|=q-1,\ G\cap A_t\neq\emptyset,\ \forall t}{u/x_1\geq_{lex} x_G}P_{G^c}\cap$$
$$\cap\bigcap\limits_\twoline{G\subset[n]\setminus\{1\},\ |G|=n-q+1}{x_{G\setminus\min(G)}\geq_{lex}x_{F^c\setminus\{1\}}}P_{G}\cap\bigcap\limits_\twoline{P\in \Min(L^f_{S}(u))}{\height P=n-q} P,\mbox{ otherwise.}$$
\end{Theorem}

\begin{proof}
To begin with, we describe the facets of the simplicial complex $\Delta$ associated with $I$ which have cardinality greater than $q$. 

Let $G$ be a facet of $\Delta$ with $|G|>q$. Then $x_G\notin I$ and for all $i\in[n]\setminus G$, $x_{G\cup\{i\}}\in \shad^e(I)$, for some $e>0$, where $\shad^e(I)$ is the $e-$th shadow of $I$. Since $I$ is a completely squarefree lexsegment ideal, $\shad(I)=\shad(L^i(v))\cap\shad(L^f(u))=\shad(L^i(v))\cap\Mon_{q+1}^s(S)=\shad(L^i(v))$. 

We also note that since $|G|>q$, then $x_G\in\shad^e(L^f(u))$, thus $x_G\in (L^f(u))$. Therefore $G$ is a facet of $\Delta$ if and only if $x_G\notin(L^i(v))$ and $x_{G\cup\{i\}}\in \shad^e(L^i(v))$, for some $e>0$, which is equivalent to the fact that $G$ is a facet of $\Gamma_1$, where $\Gamma_1$ is the simplicial complex associated with $(L^i(v))$. By Theorem \ref{prim dec initial}, it follows that $G=[n]\setminus A_t$, for some $1\leq t\leq q$, where $A_t=[j_t]\setminus\{j_1,\ldots,j_{t-1}\}$ and $|A_t|<n-q$.

In the second step of the proof, we describe the facets of $\Delta$ of cardinality $q$. Let $G$ be a facet of $\Delta$ with $|G|=q$. Then $x_G\notin I$ and $x_{G\cup\{i\}}\in I$, for all $i\in[n]\setminus G$, since, as we noticed above, $\shad(L^f(u))=\Mon_{q+1}^s(S)$. Therefore $G$ is a facet of $\Delta$ with $|G|=q$ if and only if $x_G\notin G(I)=L(u,v)$ and $x_{G\cup \max([n]\setminus G)}\geq_{lex}vx_{\max([n]\setminus\supp(v))}$. In other words, $G$ is a facet of $\Delta$ if and only if $G$ is a facet of $\Gamma_2$, where $\Gamma_2$ is the simplicial complex associated with $(L^f(u))$, which is equivalent to $P_{G^c}\in\Min(L^f(u))$, or $x_G<_{lex}v$ and $x_{G\cup \max([n]\setminus G)}\geq_{lex}vx_{\max([n]\setminus\supp(v))}$. The later condition says that $G$ is a facet of $\Gamma_1$, which means that $G=[n]\setminus A_t$, for some $t$ such that $|A_t|=n-q$. 

Finally, let us describe the facets $G$ of $\Delta$ with $|G|=q-1$. We have that $G$ is a facet of $\Delta$ if and only if $x_{G\cup\{i\}}\in I$, for all $i\in[n]\setminus G$, hence if and only if $G$ is a facet of $\Gamma_1$ and $x_{G\cup \min([n]\setminus G)}\leq_{lex}u$. 

We have to consider the following cases:

\textit{Case 1:} There is an integer $1\leq t\leq q$ such that $G=[n]\setminus A_t$, that is $G=\{j_1,\ldots,j_{t-1}\}\cup\{j_t+1,\ldots,n\}$. Then we obtain $\min([n]\setminus G)=1$ and the condition $x_{G\cup \min([n]\setminus G)}\leq_{lex}u$ is equivalent to $x_G\leq_{lex}u/x_1$.

Since $|G|=q-1$, it follows that $j_t-(t-1)=n-q+1$, that is $j_t=n-q+t$. In this case, $v=x_{j_1}\cdots x_{j_{t-1}}x_{j_t}x_{j_t+1}\cdots x_n$, thus $x_G=v/x_{j_t}$. Therefore, $G=[n]\setminus A_t$ is a facet of $\Delta$ of cardinality $q-1$ if and only if $u/x_1\geq_{lex} v/x_{j_t}$. 

\textit{Case 2:} Let $G$ be a facet of $\Gamma_1$ with $|G|=q-1$ and $G\cap A_t\neq\emptyset$ for all $1\leq t\leq q$, such that $x_{G\cup \min([n]\setminus G)}\leq_{lex}u$. If $1\notin G$, then $\min([n]\setminus G)=1$, hence $x_G\leq_{lex}u/x_1$. 

Let now consider $1\in G$ and assume that $i_2\geq 3$. In this case, the condition $x_{G\cup \min([n]\setminus G)}\leq_{lex}u$ is equivalent to $x_{G\setminus\{1\}}x_{\min([n]\setminus G)}\leq_{lex}u/x_1=x_{i_2}\cdots x_{i_q}$, which is imposible since $x_2\mid x_{G\setminus\{1\}}x_{\min([n]\setminus G)}$. Therefore, in this case, the proof is completed.

What is left is to consider $1\in G$ and $i_2=2$. Then $G$ must satisfy the following conditions: $G\cap A_t\neq \emptyset$, for all $1\leq t\leq q$  and $x_{([n]\setminus G)\setminus \{\min([n]\setminus G)\}}\geq_{lex}x_{F^c\setminus \{1\}}$, the later one being equivalent to $x_{G\setminus\{1\}}x_{\min([n]\setminus G)}\leq_{lex}u/x_1$.  
\end{proof}

Next, we give some immediate consequences.% of Theorem \ref{prim dec completely} and Proposition \ref{multiplicity}. 

\begin{Corollary}
If $I$ is a completely squarefree lexsegment ideal determined by the monomials $u=x_1x_{i_2}\cdots x_{i_q}$, $2\leq i_2<\ldots<i_q\leq n$ and $v=x_{j_1}\cdots x_{j_q}$, $2\leq j_1<\ldots<j_q\leq n$, then $\dim(S/I)=n-j_1$.
\end{Corollary}

\begin{Corollary}
Let $I\subset S$ be a completely squarefree lexsegment ideal determined by the monomials $u=x_1x_{i_2}$, $u\neq x_1x_2$, and $v=x_{j_1}x_{j_2}$, $v\neq x_{n-1}x_{n}$, and $G$ be the graph with the edge ideal $I$. Then 
$$I=P_{A_1}\cap P_{A_2}\cap\left(\bigcap\limits_{i_2\leq s<j_1}P_{[n]\setminus\{s\}}\right)\cap\left(\bigcap\limits_{s<i_2}P_{[n]\setminus\{1,s\}}\right),\mbox{ if }j_2\leq n-1,$$
and 
$$I=P_{A_1}\cap\left(\bigcap\limits_{i_2\leq s\leq j_1}P_{[n]\setminus\{s\}}\right)\cap\left(\bigcap\limits_{s<i_2}P_{[n]\setminus\{1,s\}}\right),\mbox{ if }j_2=n,$$
where $A_1=[j_1]$ and $A_2=[j_2]\setminus\{j_1\}$.

%The sets $A_1$, $A_2$ with $|A_2|\leq n-q+1$, where for the equality we impose the condition $i_2\leq j_1$, the sets $P_{[n]\setminus\{s\}}$, for all $i_2\leq s <j_1$ and the sets $P_{[n]\setminus\{1,s\}}$, with $s<i_2$ are the minimal vertex covers of $G$.
\end{Corollary}
From the above formula one derives the minimal vertex covers of $G$.

%Note that we considered only the case $i_2>2$, since otherwise, for $i_2=2$, the ideal $I$ is an initial squarefree lexsegment ideal and the minimal vertex covers are given in Corollary \ref{min vertex cover initial}.

\begin{Corollary}
Let $I\subset S$ be a completely squarefree lexsegment ideal determined by the monomials $u=x_1x_{i_2}\cdots x_{i_q}=x_1x_F$, $2\leq i_2<\ldots<i_q\leq n$ and $v=x_{j_1}\cdots x_{j_q}$, $2\leq j_1<\ldots<j_q\leq n$, and $\Delta$ be the simplicial complex with the Stanley-Reisner ideal $I$. Let $s$ be the unique integer such that $j_i=j_1+(i-1)$, for all $1\leq i< s$ and $j_s\geq j_1+s$, and $t$ be the cardinality of the set $\{x_G: x_{F^c\setminus\{1\}}>_{lex}x_G, |G|=n-q\}$. Then the multiplicity of $k[\Delta]$ is 
	\[e(k[\Delta])=\left\{\begin{array}{cc}
			s-1&,\mbox{ if }j_1<n-q \\
			s+t-1 &,\mbox{ if }j_1=n-q.   \\
	\end{array}\right.\]
\end{Corollary}

\begin{proof}
The multiplicity of $k[\Delta]$ is $e(k[\Delta])=f_{n-j_1-1}$. To determine the number of facets of $\Delta$ of cardinality $n-j_1$ is equivalent to determine the number of minimal prime ideals which contains $I$ of $\height(p)=j_1$. Let $s$ be the unique integer such that $j_i=j_1+(i-1)$, for all $1\leq i< s$ and $j_s\geq j_1+s$. Then $|A_i|=j_i-(i-1)=j_1$, for all $1\leq i\leq s-1$ and $|A_i|=j_i-(i-1)>j_1$, for $i\geq s$. By Theorem \ref{prim dec completely}, if $j_1<n-q$, then $e(k[\Delta])=s-1$, and if $j_1=n-q$, then $e(k[\Delta])=s-1+t$, where $t$ is the cardinality of the set $\{x_G: x_{F^c\setminus\{1\}}>_{lex}x_G, |G|=n-q\}$.
\end{proof}

\section{Completely squarefree lexsegment ideals which are sequentially Cohen--Macaulay}

This section is devoted to determining the completely squarefree lexsegment ideals which are sequentially Cohen--Macaulay. 

For an ideal $I$, we will denote by $I^\vee$ its Alexander dual. By \cite [Theorem 8.2.20]{HeHi}, to prove that the ideal $I$ is sequentially Cohen--Macaulay is equivalent to prove that its Alexander dual is componentwise linear. If $I\subset S$ is a graded ideal, then we denote by $I_{\langle j\rangle}$ the ideal generated by all homogeneous polynomials of degree $j$ belonging to $I$. A graded ideal $I\subset S$ is componentwise linear if $I_{\langle j\rangle}$ has a linear resolution for all $j$.

As we did for the primary decomposition, we will start the study of the sequentially Cohen--Macaulay property with the case of initial squarefree lexsegment ideals. Moreover, we will prove a more general result, which provides a new class of componentwise squarefree ideals. 

In \cite{MH}, the so-called canonical critical ideals have been studied. We recall the definition. A homogeneous ideal $I\subset S$ is called \textit{canonical critical} if it is of the form $I=(f_1x_1,f_1f_2x_2,\ldots,f_1\cdots f_{s-1}x_{s-1},f_1f_2\cdots f_s)$, for some homogeneous polynomials $f_1,\ldots,f_s$, with $f_i\in k[x_i,\ldots ,x_n]$, for each $1\leq i\leq s$ and with $\deg(f_s)>0$, where $1\leq s\leq n$. 

We consider now the class of squarefree monomial (canonical) critical ideals. Let $S=k[x_1,\ldots,x_n]$ be the polynomial ring in $n$ variables over a field $k$ and $I$ be a squarefree monomial ideal in $S$. We denote by $G(I)$ the minimal system of monomial generators of $I$.

\begin{Definition}\rm
A squarefree monomial ideal $I\subset S$ is called \textit{critical} if it is obtained by the following recursive procedure:
\begin{enumerate}
	\item [(a)] The ideal $I$ is of the form $I=(x_i,m)$, for some squarefree monomial $m$ with $x_i\nmid m$,
	\item [(b)] There is a variable $x_j$, a squarefree monomial $m'$ and $J$ a critical ideal, with $x_j\nmid mm'$ and $\supp(m)\cap\supp(m')=\emptyset$, for all the monomials $m\in G(J)$, such that $I=(x_j)+m'J$.
\end{enumerate}
\end{Definition}

\begin{Definition}\rm
A squarefree monomial ideal $I\subset S$ is called \textit{canonical critical} if there exists $w$ a monomial in $S$ such that $I=wJ$, where $J$ is a critical squarefree monomial ideal.
\end{Definition}

\begin{Proposition} 
Let $I\subset S$ be a critical squarefree monomial ideal. Then $I$ has linear quotients.
\end{Proposition}

\begin{proof}
We use induction on $|G(I)|$.

If $|G(I)|=2$, then $I=(x_i,m)$, for some squarefree monomial $m$ and $x_i\nmid m$. Since $(x_i):m=(x_i)$, it follows that $I$ has linear quotients. 
 
Assume that the assertion holds for all critical ideals $J$, with $|G(J)|=s$, with $s\geq 2$. By definition, there is a variable $x_j$, a squarefree monomial $m'$ and $J$ a critical ideal, with $x_j\nmid mm'$ and $\supp(m)\cap\supp(m')=\emptyset$, for all the monomials $m\in G(J)$, such that $I=(x_j)+m'J$ and $|G(I)|=s+1$. 

Let us assume that $G(J)=\{x_i,m_2,\ldots, m_s\}$ such that $J$ has linear quotients with respect to $x_i,m_2,\ldots, m_s$. We order the minimal system of generators of $I$, $G(I)=\{x_j,m'x_i,m'm_2,\ldots ,m'm_s\}$. For every $2\leq t\leq s$, we have that the ideal quotient 
$$(x_j,m'x_i,\ldots ,m'm_{t-1}):(m'm_t)=(x_j):(m'm_t)+(m'x_i,\ldots ,m'm_{t-1}):(m'm_t)=$$
$$=(x_j)+(x_i,\ldots ,m_{t-1}):(m_t)$$
is generated by variables, by the induction hypothesis. Hence $I$ has linear quotients.
\end{proof}

\begin{Corollary}
Any canonical critical squarefree monomial ideal has linear quotients.
\end{Corollary}

The next result gives us a new class of componentwise linear ideals.

\begin{Corollary}\label{critical comp lin}
If $I$ is a canonical critical ideal, then $I$ is componentwise linear. 
\end{Corollary}

\begin{proof}
This is a consequence of \cite{JZ}.
\end{proof}

We will use the properties of these ideals for the case of initial squarefree lexsegment ideals.

\begin{Proposition}
Let $I\subset k[x_1,\ldots,x_n]$ be the initial squarefree lexsegment ideal, generated in degree $q$, determined by the monomial $v=x_{j_1}\cdots x_{j_q}$, with $2\leq j_1<\cdots<j_q\leq n$. Then $I$ is sequentially Cohen--Macaulay.
\end{Proposition}

\begin{proof}
We need to prove that the Alexander dual $I^\vee$ is componentwise linear. By the primary decomposition, Theorem \ref{prim dec initial}, one has $I^\vee=J+K$, where $J=(x_{A_t}:1\leq t\leq q)$ and $K=(x_{G^c}:|G|=q-1,G\cap A_t\neq\emptyset,1\leq t\leq q)$. Since $J$ is generated in degree at most $n-q$ and $K$ is generated in degree $n-q+1$, we obtain that $I^\vee_{\langle j \rangle}=J_{\langle j \rangle}$, for all $j<n-q+1$.

For all $j<n-q+1$, one has that $I^\vee_{\langle j \rangle}=J_{\langle j \rangle}$ and it has a linear resolution, since $J$ is a canonical critical squarefree monomial ideal. Indeed, we have 
$$J=x_1\cdots x_{j_1-1}(x_{j_1}+x_{j_1+1}\cdots x_{j_2-1}(x_{j_2}+x_{j_2+1}\cdots x_{j_3-1}(\ldots))).$$

We prove that $I^\vee_{\langle n-q+1 \rangle}=I_{n,n-q+1}$ is the ideal generated by all the squarefree monomials of degree $n-q+1$ in $S$. This will end our proof since $I_{n,n-q+1}$ has a linear resolution.

Let $m=x_{F^c}\in G(I_{n,n-q+1})$, with $|F|=q-1$, be a squarefree monomial. If $F\cap A_t\neq\emptyset$, for all $1\leq t\leq q$, then $m\in G(I^\vee_{\langle n-q+1 \rangle})$. Assume that there is an integer $1\leq t\leq q$ such that $F\cap A_t=\emptyset$. It results that $F^c\supset A_t$, hence $x_{A_t}\mid m$, thus $m\in G(I^\vee_{\langle n-q+1 \rangle})$. 

The other inclusion is trivial, namely $I^\vee_{\langle n-q+1 \rangle}\subseteq I_{n,n-q+1}$.
\end{proof}

The final squarefree lexsegment ideals are sequentially Cohen--Macaulay, as it follows from the next result.

\begin{Proposition}
Let $I\subset S$ be the final squarefree lexsegment ideal generated in degree $q$, determined by the monomial $u=x_1x_{i_2}\cdots x_{i_q}$, $2\leq i_2<\ldots<i_q\leq n$. Then $I$ is sequentially Cohen--Macaulay.
\end{Proposition}

\begin{proof}
We may assume that $I\neq I_{n,q}$. We will prove that the Alexander dual of $I$, $I^\vee$, is componentwise linear. 

It is easy to see, by Theorem \ref{prim dec final}, that $I^\vee_{\langle n-q\rangle}$ has a linear resolution, since it is a final squarefree lexsegment ideal. 

We show that $I^\vee_{\langle n-q+1\rangle}$ is the ideal generated by all the squarefree monomials of degree $n-q+1$. Indeed, we have the inclusion $I^\vee_{\langle n-q+1\rangle}\subset I_{n,n-q+1}$. For the other one, let $m=x_{G^c}$ be a squarefree monomial of degree $n-q+1$. It is clear that, using the notation $u=x_1x_F$, if $m=x_{G^c}\geq_{lex}x_{F^c}$, then $m\in G(I^\vee_{\langle n-q+1\rangle})$. Otherwise, assume that $m=x_{G^c}<_{lex}x_{F^c}$. We have to analyze two cases:

\textit{Case 1:} If $x_{F^c\setminus\{1\}}>_{lex}x_{G^c\setminus\min(G^c)}$, then $x_{G^c\setminus\min(G^c)}\in G(I^\vee_{\langle n-q\rangle})$ and $m=x_{G^c}\in I^\vee_{\langle n-q+1\rangle}$.

\textit{Case 2:} If $x_{F^c\setminus\{1\}}\leq_{lex}x_{G^c\setminus\min(G^c)}$, then $m=x_{G^c}\in G(I^\vee_{\langle n-q+1\rangle})$, which ends the proof. 
\end{proof}

We now focus on arbitrary completely squarefree lexsegment ideals. In order to characterize the completely squarefree lexsegment ideals which are sequentially Cohen--Macaulay, we have to establish when $I^\vee_{\langle j\rangle}$ has a linear resolution, for every $j\leq n-q+1$. Analyzing the minimal primary decompositions obtained if $x_2\mid u$ or $x_2\nmid u$, one can see that
\begin{eqnarray}
I^\vee_{\langle j\rangle}=(x_{A_t}:|A_t|\leq n-q)_{\langle j\rangle}\mbox{ for all }j<n-q\mbox{ and } \eqname{$1$}\label{1}
\end{eqnarray}
\begin{eqnarray}
I^\vee_{\langle n-q\rangle}=(x_{A_t}:|A_t|\leq n-q)_{\langle n-q\rangle}+(x_G:x_{F^c\setminus\{1\}}>_{lex}x_G,|G|=n-q), \eqname{$2$}\label{2}
\end{eqnarray}
where $u=x_1x_F$ and $A_t=[j_t]\setminus\{j_1,\ldots,j_{t-1}\},\mbox{ for }1\leq t\leq q.$ Note that only for $I^\vee_{\langle n-q+1\rangle}$ we have to treat separate cases, given by the conditions $x_2\mid u$ or $x_2\nmid u$. 

\begin{Lemma}\label{dual lower deg}
For every $j<n-q$, the ideal $I^\vee_{\langle j\rangle}$ has a linear resolution.
\end{Lemma}

\begin{proof}
The ideal $I^\vee_{\langle j\rangle}$ has a linear resolution, for every $j<n-q$, since it is a canonical critical squarefree monomial ideal. 
\end{proof}

In order to determine when
$$I^\vee_{\langle n-q\rangle}=(x_{A_t}:|A_t|\leq n-q)_{\langle n-q\rangle}+(x_G:x_{F^c\setminus\{1\}}>_{lex}x_G,|G|=n-q)$$
has a linear resolution, we need some preparatory results.

Firstly, let $1\leq s\leq q$ be the unique index with the property that $j_s\leq n-q+s-1$ and $j_{s+1}>n-q+1$. Then we have $(x_{A_t}:|A_t|\leq n-q)_{\langle n-q\rangle}=(x_{A_1},x_{A_2},\ldots ,x_{A_s})_{\langle n-q\rangle}$.

\begin{Proposition}\label{dual deg n-q}
The ideal $(x_{A_t}:|A_t|\leq n-q)_{\langle n-q\rangle}$ is an initial squarefree lexsegment ideal determined by the monomial $x_{A_s}x_{q+j_s-s+2}\cdots x_n$. 
\end{Proposition}

\begin{proof}We start by noticing that if the cardinality of $A_s$ equals $n-q$, then the monomial $x_{A_s}x_{q+j_s-s+2}\cdots x_n=x_{A_s}$.

For the inclusion "$\subseteq$", let $m$ be a minimal monomial generator of $(x_{A_t}:|A_t|\leq n-q)_{\langle n-q\rangle}=(x_{A_1},x_{A_2},\ldots ,x_{A_s})_{\langle n-q\rangle}$, that is $m=x_{A_t}m'$, for some $1\leq t\leq s$ with $A_t\cap \supp(m')=\emptyset$ and we assume that $t$ is the smallest with this property. We have to prove that $m\geq_{lex} x_{A_s}x_{q+j_s-s+2}\cdots x_n$. 

Assume by contradiction that $x_{A_s}x_{q+j_s-s+2}\cdots x_n>_{lex}m=x_{A_t}m'$. We observe that $x_{l}\mid x_{A_t}m'$ and $x_{l}\mid x_{A_s}x_{q+j_s-s+2}\cdots x_n$, for all $l\in A_t$, $l\neq j_t$, thus if $s\neq t$, then we reach a contradiction. Therefore, we must have $t=s$ and by the relation $x_{A_s}x_{q+j_s-s+2}\cdots x_n>_{lex}m=x_{A_t}m'$ we get $x_{q+j_s-s+2}\cdots x_n>_{lex}m'$. This is a contradiction, hence $m\geq_{lex} x_{A_s}x_{q+j_s-s+2}\cdots x_n$.

Conversely, for the inclusion "$\supseteq$", let $m\geq_{lex} x_{A_s}x_{q+j_s-s+2}\cdots x_n$ be a squarefree monomial of degree $n-q$. We claim that $\supp(m)\cap\supp(v)\neq\emptyset$. Indeed, if we assume that $\supp(m)\cap\supp(v)=\emptyset$, by degree considerations, we obtain that $m=x_{[n]\setminus\supp(v)}=x_{\{1,\ldots n\}\setminus\{j_1,\ldots,j_q\}}$. Hence  
$$x_{A_s}x_{q+j_s-s+2}\cdots x_n=x_{\{1,\ldots,j_s\}\setminus\{j_1,\ldots,j_{s-1}\}}x_{q+j_s-s+2}\cdots x_n>_{lex}x_{\{1,\ldots n\}\setminus\{j_1,\ldots,j_q\}}=m,$$
contradiction. 

Therefore $\supp(m)\cap\supp(v)\neq\emptyset$. Denote by $j_t=\min\{i\in[n]:i\in \supp(m)\cap\supp(v)\}$. We claim that 
$$\{r\in \supp(m): r\leq j_t\}=[j_t]\setminus\{j_1,\ldots,j_{t-1}\}=A_t.$$

Indeed, by hypothesis, we have 
$$m=x_{\{r\in \supp(m): r\leq j_t\}}x_{\{r\in \supp(m): r> j_t\}}\geq_{lex} x_{A_s}x_{q+j_s-s+2}\cdots x_n=$$
$$=x_{\{1,\ldots,j_t\}\setminus\{j_1,\ldots,j_{t-1}\}}x_{\{j_t,\ldots,j_s\}\setminus\{j_t,\ldots,j_{s-1}\}}x_{q+j_s-s+2}\cdots x_n.$$
Since $j_t$ is minimal, we get $\{r\in \supp(m): r\leq j_t\}\cap\{j_1,\ldots,j_{t-1}\}=\emptyset$. Hence $x_{\{r\in \supp(m): r\leq j_t\}}=x_{[j_t]\setminus\{j_1,\ldots,j_{t-1}\}}$. Otherwise $x_{\{r\in \supp(m): r\leq j_t\}}<_{lex}x_{[j_t]\setminus\{j_1,\ldots,j_{t-1}\}}$, thus $m<_{lex}x_{A_s}x_{q+j_s-s+2}\cdots x_n$, a contradiction.

It results that $\{r\in \supp(m): r\leq j_t\}=[j_t]\setminus\{j_1,\ldots,j_{t-1}\}=A_t$ and $x_{A_t}\mid m$, thus $m\in (x_{A_t}:|A_t|\leq n-q)_{\langle n-q\rangle}$, which ends the proof.
\end{proof}

Returning to the ideal $I^\vee_{\langle n-q\rangle}$, by (\ref{2}), it follows that $I^\vee_{\langle n-q\rangle}$ is the sum of an initial with a final squarefree lexsegment ideal, both of them generated in the same degree. In the following, in order to determine the ideals $I^\vee_{\langle n-q\rangle}$ with a linear resolution, we will analyze a more general problem.

\medskip
We consider $J=(L^i(w))$ and $K=(L^f(m))$ be initial and final squarefree lexsegment ideals, generated in degree $d$, such that $x_1\mid w$ and $x_1\nmid m$. The ideals $J$ and $K$ have a $d-$linear resolution. We are interested in determining when the sum $J+K$ has a $d-$linear resolution.

\begin{Proposition}\label{intersection}
In the above hypotheses, the ideal $J+K$ has a $d-$linear resolution if and only if the ideal $J\cap K$ has a $(d+1)-$linear resolution.
\end{Proposition}

\begin{proof}
For the implication "$\Leftarrow$", assume that $J\cap K$ is generated in degree $d+1$ and has a linear resolution.
We use the following exact sequence of $S-$modules
$$0\rightarrow J\cap K\rightarrow J\oplus K\rightarrow J+K\rightarrow 0,$$
where $J\cap K$ has a $(d+1)-$linear resolution and $J\oplus K$ has a $d-$linear resolution. By \cite [Lemma 3.2]{O}, we obtain that $J+K$ has a $d-$linear resolution.

For the converse implication "$\Rightarrow$", assume that $J+K$ has a $d-$linear resolution. We want to prove that $J\cap K$ is generated in degree $d+1$ and has a linear resolution. The exact sequence
$$0\rightarrow J\cap K\rightarrow J\oplus K\rightarrow J+K\rightarrow 0,$$
yields the exact sequence 
$$\ldots\rightarrow \Tor_1(J+K,k)\rightarrow \Tor_0(J\cap K,k)\rightarrow \Tor_0(J,k)\oplus \Tor_0(K,k)\rightarrow\ldots.$$
This implies $J\cap K$ is generated in degree $d$ or $d+1$. Taking into account that every minimal monomial generator of $J$ is divisible by $x_1$ and $K\subset k[x_2,\ldots,x_n]$, it results that $J\cap K$ is generated in degree $d+1$.

Next, we prove that $J\cap K$ has a linear resolution, that is $\Tor_i(J\cap K,k)_{i+j}=0$, for all $j\neq d+1$. We consider the exact sequence
$$\ldots\rightarrow \Tor_i(J+K,k)_{i+j}\rightarrow \Tor_{i-1}(J\cap K,k)_{i+j}\rightarrow \Tor_{i-1}(J,k)_{i+j}\oplus \Tor_{i-1}(K,k)_{i+j}\rightarrow\ldots$$
For every $j\neq d$ we have $\Tor_{i-1}(J,k)_{i+j}\oplus \Tor_{i-1}(K,k)_{i+j}=0$, since $J$ and $K$ have a linear resolution, and $\Tor_i(J+ K,k)_{i+j}=0$. Hence $\Tor_{i-1}(J\cap K,k)_{i+j}=0$, for all $j\neq d$, that is $\Tor_{i-1}(J\cap K,k)_{(i-1)+(j+1)}=0$, for all $j\neq d$.
\end{proof}

\begin{Lemma}\label{intersection d+1}
In the same hypothesis, the ideal $J\cap K$ is generated in degree $d+1$ if and only if $J\cap K=(L(x_1m,wx_{\max([n]\setminus\supp(w))}))$.
\end{Lemma}

\begin{proof}
If $J\cap K$ is generated in degree $d+1$, then it is easy to see that $G(J\cap K)=L(x_1m,wx_{\max([n]\setminus\supp(w))})$. Indeed, $G(J\cap K)\subseteq L^i(wx_{\max([n]\setminus\supp(w))})\cap L^f(x_1m)=L(x_1m,wx_{\max([n]\setminus\supp(w))})$. For the other inclusion, let $\gamma\in L(x_1m,wx_{\max([n]\setminus\supp(w))})$. It is clear that $x_1\mid \gamma$. We obtain that $m\geq \gamma/x_1$, that is $\gamma\in K$. Moreover, $\gamma\in J$ because $\gamma/x_{\max(\gamma)}\geq_{lex}w$. Hence $\gamma\in J\cap K$, thus $\gamma\in G(J\cap K)$, as desired.

The converse is clear.
\end{proof}

We return to the problem of characterizing the completely squarefree lexsegment ideals which are sequentially Cohen--Macaulay.

Let $I\subset S$ be a completely squarefree lexsegment ideal generated in degree $q$, determined by the monomials $u=x_1x_{i_2}\cdots x_{i_q}=x_1x_F$, $2\leq i_2<\ldots<i_q\leq n$ and $v=x_{j_1}\cdots x_{j_q}$, with $2\leq j_1<\ldots<j_q\leq n$, which is not an initial or a final squarefree lexsegment ideal. Denote by $A_t=[j_t]\setminus\{j_1,\ldots,j_{t-1}\}$ for $1\leq t\leq q$. Let $s$ be the unique index such that $|A_s|\leq n-q$ and $|A_{s+1}|=n-q+1$.

Recall that if $x_2\nmid u$, then 
$$I^\vee_{\langle n-q+1\rangle}=(x_{A_t}:|A_t|\leq n-q)_{\langle n-q+1\rangle}+(x_{A_t}:|A_t|=n-q+1, u/x_1\geq_{lex} v/x_{j_t})+$$
$$+(x_{G^c}:G\subset[n],\ |G|=q-1,\ G\cap A_t\neq\emptyset,1\leq t\leq q,u/x_1\geq_{lex} x_G)+$$
$$+(x_G: x_{F^c\setminus\{1\}}>_{lex}x_G,|G|=n-q)_{\langle n-q+1\rangle}$$
and if $x_2\mid u$, then 
$$I^\vee_{\langle n-q+1\rangle}=(x_{A_t}:|A_t|\leq n-q)_{\langle n-q+1\rangle}+(x_{A_t}:|A_t|=n-q+1, u/x_1\geq_{lex} v/x_{j_t})+$$
$$+(x_{G^c}:G\subset[n]\setminus\{1\},\ |G|=q-1,\ G\cap A_t\neq\emptyset,1\leq t\leq q,u/x_1\geq_{lex} x_G)+$$
$$+(x_G:x_{F^c\setminus\{1\}}>_{lex}x_G,|G|=n-q)_{\langle n-q+1\rangle}+(x_G:x_{G\setminus\min(G)}\geq_{lex}x_{F^c\setminus\{1\}},|G|=n-q+1).$$
%Moreover, we consider $u'=\succ(x_{F^c})$ and $v'=x_{A_s}x_{q+j_s-s+1}\cdots x_n$, that is $v'=\left(\prod\limits_{\twoline{1\leq i\leq j_s}{i\notin\supp(v/x_{j_s})}}x_i\right)\left(\prod\limits_{i=q+j_s-s+1}^{n}x_i\right)$.

\begin{Theorem}
In the above hypotheses, $I$ is sequentially Cohen--Macaulay if and only if $(L^i(x_{A_s}x_{q+j_s-s+2}\cdots x_n))\cap(L^f(\succ(x_{F^c\setminus\{1\}})))$ has a $(n-q+1)-$linear resolution.
\end{Theorem}

\begin{proof}
We have to establish when the Alexander dual $I^\vee$ is componentwise linear. By Lemma \ref{dual lower deg}, $I^\vee_{\langle j\rangle}$ has a linear resolution, for all $j<n-q$.

By Proposition \ref{dual deg n-q}, we have $I^\vee_{\langle n-q\rangle}=(L^i(x_{A_s}x_{q+j_s-s+2}\cdots x_n))+(L^f(\succ(x_{F^c\setminus\{1\}})))$. Hence $I^\vee_{\langle n-q\rangle}$ has a linear resolution if and only if, using Proposition \ref{intersection}, 
$$(L^i(x_{A_s}x_{q+j_s-s+2}\cdots x_n))\cap(L^f(\succ(x_{F^c\setminus\{1\}})))$$
has a $(n-q+1)-$linear resolution. 

We end the proof by showing that $I^\vee_{\langle n-q+1\rangle}=I_{n,n-q+1}$. 

\textit{Case 1.} We assume that $x_2\nmid u$. It is clear the inclusion $I^\vee_{\langle n-q+1\rangle}\subseteq I_{n,n-q+1}$. For the other one, let $m=x_{H^c}$ be a squarefree monomial of degree $n-q+1$. If $x_{F^c}>_{lex}m$, then $m\in (L^f(\succ(x_{F^c})))=(L^f(\succ(x_{F^c\setminus\{1\}})))_{\langle n-q+1\rangle}$.

Assume that $m=x_{H^c}\geq_{lex}x_{F^c}$, equivalently $x_H\leq_{lex} x_{F}$. If $H\cap A_t\neq\emptyset$ for all $1\leq t\leq q$, then $m\in (x_{G^c}: |G|=q-1, G\subset[n],G\cap A_t\neq\emptyset,1\leq t\leq q,u/x_1\geq_{lex} x_G)\subseteq I^\vee_{\langle n-q+1\rangle}$. Otherwise, if there is some $t$ such that $H\cap A_t=\emptyset$, then we must have $j_t-t+1+q-1\leq n$, that is $j_t-t+1\leq n-q+1$. If $j_t-t+1\leq n-q$, then $x_{A_t}\in (x_{A_t}:|A_t|\leq n-q)$. Moreover, by $H\cap A_t=\emptyset$ we obtain $A_t\subseteq H^c$, thus $x_{A_t}\mid m$ and $m\in (x_{A_t}:|A_t|\leq n-q)_{\langle n-q+1\rangle}$. If $j_t-t+1=n-q+1$, then $v=x_{j_1}\cdots x_{j_{t-1}}x_{n-q+t}\cdots x_n$. In particular, $H^c=A_t$ and by the relation $m=x_{H^c}=x_{A_t}\geq_{lex}x_{F^c}$ we get $u/x_1=x_F\geq_{lex} v/x_{j_t}$. Hence $m\in (x_{A_t}: |A_t|=n-q+1,\ u/x_1\geq_{lex} v/x_{j_t})\subseteq I^\vee_{\langle n-q+1\rangle}$, which ends the proof.

\textit{Case 2.} For the case when $x_2\mid u$, we have the following. Let $m=x_{H^c}$ be a squarefree monomial of degree $n-q+1$. 

We analyze separately the cases when $x_1\mid m$ and $x_1\nmid m$.

Firstly, if $x_1\nmid m$, then either $x_{H^c\setminus\min(H^c)}\geq_{lex}x_{F^c\setminus\{1\}}$, thus $m=x_{H^c}\in G(I^\vee_{\langle n-q+1\rangle})$, either $x_{H^c\setminus\min(H^c)}<_{lex}x_{F^c\setminus\{1\}}$ which implies $x_{H^c\setminus\min(H^c)}\in G(I^\vee_{\langle n-q\rangle})$, thus $x_{H^c}\in I^\vee_{\langle n-q+1\rangle}$. 

Assume that $m$ is divisible by $x_1$. By the minimal primary decomposition, the following cases remains to study:
\begin{itemize}
	\item $H\cap A_t\neq\emptyset$ for all $1\leq t\leq q$ and $u/x_1<_{lex} x_H$;
	\item there exists an integer $1\leq t\leq q$ such that $H\cap A_t=\emptyset$,
\end{itemize}
%since for the case when $H\cap A_t\neq\emptyset$ for all $1\leq t\leq q$ and $u/x_1\geq_{lex} x_H$, we have $m\in I^\vee_{\langle n-q+1\rangle}$.

If we assume that $H\cap A_t\neq\emptyset$ for all $1\leq t\leq q$ and $u/x_1=x_F<_{lex} x_H$, then $m=x_{H^c}<_{lex}x_{F^c}$. Since $x_1\mid x_{H^c}$, we obtain $x_{H^c\setminus\{1\}}\in (L^f(\succ(x_{F^c\setminus\{1\}})))$, thus $x_{H^c}\in I^\vee_{\langle n-q+1\rangle}$.

For the case when there is some $1\leq t\leq q$ such that $H\cap A_t=\emptyset$, we obtain that $|H|+|A_t|\leq n$. This implies that $j_t-t+1\leq n-q+1$. If $|A_t|=j_t-t+1\leq n-q$, then $x_{A_t}\in(x_{A_t}:|A_t|\leq n-q)$ and $x_{A_t}\mid x_{H^c}$, because $A_{t}\subset H^c$. Thus $m=x_{H^c}\in (x_{A_t}:|A_t|\leq n-q)_{\langle n-q+1\rangle}$. Finally, if $|A_t|=j_t-t+1=n-q+1$, then $v=x_{j_1}\cdots x_{j_{t-1}}x_{n-q+t}\cdots x_n$ and $H^c=A_t$. Moreover, if $u/x_1=x_F\geq_{lex} v/x_{j_t}$, then $m=x_{A_t}\in(x_{A_t}:|A_t|=n-q+1, u/x_1\geq_{lex} v/x_{j_t})$. Otherwise, if $u/x_1=x_F<v/x_{j_t}$, we obtain $x_{F^c}>_{lex} x_{A_t}=x_{H^c}$ which implies $x_{H^c\setminus\{1\}}\in G(I^\vee_{\langle n-q\rangle})$, thus $m=x_{H^c}\in I^\vee_{\langle n-q+1\rangle}$.
\end{proof}

\section{Bounds for $\depth(S/I)$}

In this section we will give lower and upper bounds for $\depth(S/I)$, where $I\subset S$ is an arbitrary squarefree lexsegment ideal. 

We begin with a very useful lemma, which gives us an equivalent condition of the property of an ideal to have a linear resolution. 

\begin{Lemma}\label{lex ideal complementary}
Let $I=(L(x_G,x_H))\subset S$ be a squarefree lexsegment ideal, with $|G|=|H|$. %Denote by $G^c=\{1,\ldots,n\}\setminus G$ the complementary of the set $G$.
Then $I$ has a linear resolution if and only if the ideal $(L(x_{[n]\setminus H},x_{[n]\setminus G}))$ has a linear resolution.
\end{Lemma}

\begin{proof}
We prove that $S/(L(x_G,x_H))$ and $S/(L(x_{[n]\setminus H},x_{[n]\setminus G}))$ are isomorphic as modules. To this purpose, it is enough to demonstrate that if we have a  relation between the minimal monomial generators of $(L(x_G,x_H))$, then we obtain the same relation in $(L(x_{[n]\setminus H},x_{[n]\setminus G}))$. 

Let $x_G\geq_{lex} m_1>_{lex}m_2\geq_{lex} x_H$ be two minimal monomial generators of $I$ and let $x_{F_1}m_1-x_{F_2}m_2=0$ be a relation. Denote $m_1=x_{T_1}$ and $m_2=x_{T_2}$. Then the relation $x_{F_1}m_1=x_{F_2}m_2$ is equivalent to $x_{[n]\setminus (T_{1}\cup F_1)}=x_{[n]\setminus (T_2\cup F_2)}$. Equivalently, $x_{F_2}x_{[n]\setminus T_1}=x_{F_1}x_{[n]\setminus T_2}$, which give us the same relation between the minimal monomial generators $x_{[n]\setminus T_1}$ and $x_{[n]\setminus T_2}$ of $(L(x_{[n]\setminus H},x_{[n]\setminus G}))$. In particular, they have the same linear relations. 
\end{proof}

Next, we will give upper and lower bounds for the depth of an arbitrary squarefree lexsegment ideal.

Let $I$ be an arbitrary squarefree lexsegment ideal, not necessarily complete. As it follows from the proof of Corollary \ref{depth finale}, one can see that $\depth(S/I)\geq q-1$. The following result will characterize the ideals with $\depth(S/I)>q-1$.

Denote $\succ(m)=\max\{w: m>_{lex}w\}$ and $\pred(m)=\min\{w: w>_{lex}m\}$.

\begin{Proposition}
Let $I=(L(u,v))$ be an arbitrary squarefree lexsegment ideal, with $x_1\mid u$ and $x_1\nmid v$. Then $\depth(S/I)>q-1$ if and only if $\succ(v)/x_{\max(\succ(v))}\geq_{lex}\pred(u)/x_1$ and $(L(\succ(v)/x_{\max(\succ(v))},\pred(u)/x_1))$ has a linear resolution. 
\end{Proposition}

\begin{proof}
Let $\Delta$ be the simplicial complex such that $I=I_{\Delta}$. By Lemma \ref{CM skeleton}, one has $\depth(S/I)>q-1$ if and only if $\Delta^{(q-1)}$ is Cohen--Macaulay. Equivalently, by Eagon--Reiner theorem, $I_{(\Delta^{(q-1)})^\vee}$ has a linear resolution. We denote by $\Gamma=\Delta^{(q-1)}$ and by $\Gamma_1=\langle\{1,s_2,\ldots ,s_q\}: x_1x_{s_2}\cdots x_{s_q}>_{lex}u\rangle$ and $\Gamma_2=\langle\{t_1,\ldots ,t_q\}: v>_{lex}x_{t_1}\cdots x_{t_q}\rangle$. It is clear that $\Gamma=\Gamma_1\cup\Gamma_2$, equivalent, by \cite{O1}, to $I_{\Gamma^\vee}=I_{\Gamma_1^\vee}+I_{\Gamma_2^\vee}$. Since $I_{\Gamma_1^\vee}=I(\Gamma_1^c)$, we obtain that $I_{\Gamma_1^\vee}=(L^f(\succ(x_{[n]\setminus\supp(u)}))$. In a similar way, one has $I_{\Gamma_2^\vee}=(L^i(\pred(x_{[n]\setminus\supp(v)}))$. 

Then $I_{\Gamma^\vee}=I_{\Gamma_1^\vee}+I_{\Gamma_2^\vee}$ has a linear resolution if and only if, by Proposition \ref{intersection}, the ideal $I_{\Gamma_1^\vee}\cap I_{\Gamma_2^\vee}$ has an $(n-q+2)-$linear resolution. Equivalently, by Lemma \ref{intersection d+1}, $(L(x_1\succ(x_{[n]\setminus\supp(u)}),\pred(x_{[n]\setminus\supp(v)})x_{\max([n]\setminus([n]\setminus\supp(v)))})$ has an $(n-q+2)-$linear resolution. By Lemma \ref{lex ideal complementary}, this is equivalent to $(L(\succ(v)/x_{\max(\succ(v))},\pred(u)/x_1))$ has a linear resolution. 
\end{proof}

\begin{Proposition}
Let $I=(L(u,v))$ be an arbitrary squarefree lexsegment ideal with $u\neq v$ generated in degree $q$. Then $\depth(S/I)\leq n-2$ and the equality holds if and only if $I$ is Cohen--Macaulay of dimension $n-2$.
\end{Proposition}

\begin{proof}
Firstly, assume by contradiction that $\depth(S/I)=n-1$. Hence $S/I$ is Cohen--Macaulay, or equivalently, $I^\vee$ has a linear resolution. Moreover, $I^\vee$ is generated in degree $1$, since there is a minimal prime ideal of $I$ of height $1$. Then any minimal prime ideal of $I$ has height $1$, thus $I$ is a principal ideal, contradiction.

Hence $\depth(S/I)\leq n-2$. 

Let $\depth(S/I)=n-2$. If $\dim(S/I)=n-1$, then all the minimal primes of $I$ are of height $1$, thus $I$ is a principal ideal, a contradiction. Therefore $\dim(S/I)=n-2$, hence $I$ is Cohen--Macaulay.
\end{proof}

The Cohen--Macaulay squarefree lexsegment ideals have been characterized in \cite{BST}.

	\[
\]

\end{document}